\documentclass{article}
\usepackage[pagebackref,letterpaper=true,colorlinks=true,pdfpagemode=none,urlcolor=blue,linkcolor=blue,citecolor=blue,pdfstartview=FitH]{hyperref}

\usepackage{amsmath,amsfonts}
\usepackage{graphicx}
\usepackage{color}

\newif\ifblog
\newif\iftex
\blogfalse
\textrue

\usepackage{ulem}

\def\emph#1{\textit{#1}}

\newtheorem{theorem}{Theorem}
\newtheorem{lemma}[theorem]{Lemma}
\newtheorem{definition}[theorem]{Definition}

\newtheorem{proposition}[theorem]{Proposition}

\newenvironment{proof}{\noindent {\sc Proof:}}{$\Box$} %\medskip} 
\title{The density evolution of the killed Mckean-Vlasov process}
\author{
Peter E. Caines
\thanks{Department of Electrical and Computer Engineering, McGill University, Montreal,  Canada. peterc@cim.mcgill.ca }
\and{
Daniel Ho,
\thanks{Department of Mathematics, City
University of Hong Kong, Hong
Kong. madaniel@cityu.edu.hk.}
}
\and {
Qingshuo Song,
\thanks{Department of Mathematics, Worcester Polytechnic University, 
and Department of Mathematics,
City University of Hong Kong, Hong
Kong. qsong@wpi.edu.
}
}
\and{
\thanks{This research has been partially supported by GRF CityU 11201518.}
}
}

\date{}                                           % Activate to display a given date or no date

\begin{document}

\maketitle

\begin{abstract}
The study of the density evolution naturally arises in Mean Field Game theory for the estimation of the density of the large population dynamics.
In this paper, we study the density evolution of McKean-Vlasov 
stochastic differential equations in the presence of an absorbing boundary,
where the solution to such equations corresponds to the dynamics of partially killed large populations.
By using a fixed point theorem, we show that the density evolution is characterized as the unique solution of an 
integro-differential Fokker-Planck equation 
with Cauchy-Dirichlet data.
\end{abstract}

%\more

\section{Introduction}
Let $W$ be a $\mathbb R^{d}$-valued Brownian motion 
on a filtered probability space 
$(\mathbb P, \Omega, \mathcal F, \{\mathcal F_{t}\}_{t\ge 0})$, and 
we consider the stochastic differential equation
$$ d X_{t} =  \bar b(X_{t}, t) dt + dW_{t}; \ X_{0} \sim m_{0},$$
where $X_{0}$ is the given initial state with its probability density $m_{0}$ 
on $\mathbb R^{d}$.
A well known result, see  for instance Hormander's Theorem in 
Section V.38 of \cite{RW00}, 
says that the density $m(x, t) = \mathbb P(X_{t} \in dx) /dx$ satisfies
Fokker-Planck equation (FPK) with Cauchy data
$$ \left\{
\begin{array}
 {ll}
\partial_{t} m =  - {\rm div}_{x} ( \bar b \ m) + \frac 1 2 \Delta m, & (0, \infty) \times \mathbb R^{d}, \\
m(x, 0) = m_{0}(x), & x\in \mathbb R^{d}.
\end{array}\right.
$$

Recently, 
Mean Field Game theory attracted 
a great deal of attention in the control and other fields after
%great attention from the control field after 
it was 
initiated in a series of founding works  by
Huang, Caines, and Malhame (e.g., \cite{HCM03, HCM06, HCM07}), and independently in that of Lasry and Lions (e.g., \cite{LL07}). This led to 
extensive studies on the density evolution of 
McKean-Vlasov type stochastic differential equations  (MV-SDE) of 
the following form:
\begin{equation}
 \label{eq:mvsde01}
dX_{t} =  b(X_{t}, \mathbb E[X_{t}^{p}]) dt + dW_{t};\ \ X_{0} \sim m_{0}
\end{equation}
for some positive integer $p$, 
see, for instance, \cite{BFY13, Car13, GPV16} and the references therein.
It is well known that the density follows the integro-differential FPK
\begin{equation}
  \left\{
\begin{array}
 {ll}
\partial_{t} m =  %\displaystyle 
- {\rm div}_{x} ( b (x, \int_{\mathbb R^{d}} x^{p} m(x, t) dx) \ m) + \frac 1 2 \Delta m, & (0, \infty) \times \mathbb R^{d}, \\
m(x, 0) = m_{0}(x), & x\in \mathbb R^{d}.
\end{array}\right.
\label{eq:fpk03}\end{equation}

In this paper, we study a similar  integro-differential FPK associated to 
population density dynamic of the process \eqref{eq:mvsde01} 
killed at the boundary of the unit ball. 
Indeed, 
the setup of killed population has been already applied to 
mean field games in different 
contexts recently, see for instance, \cite{CMZ12}, 
\cite{CFS15}, \cite{FPC17}, and 
\cite{CF18}. 
One of the closest references to our setup 
might be
\cite{CF18}, where
FPK similar to our current study has been briefly sketched, 
see  Page 2217 \cite{CF18}. 
However, the corresponding study on the solvability and regularities of 
FPK for the killed population is not available to the best of our knowledge.

It is noted that, due to the loss of population at the boundary, 
the killed process is a strictly submarkovian in $B_{1}$. As a result, its associated FPK 
is given with Dirichlet data along the boundary 
in addition to the counter-part FPK of unkilled process \eqref{eq:fpk03}. 
Our goal is to 
legitimate the following statement: 
{\it Under appropriate conditions on the drift function $b(\cdot, \cdot)$,
the density of killed process is the classical 
solution of its associated Integro-Differential FPK with initial-boundary data}.
%The proof of the above claim, of course, shall include the unique solvabilities of \eqref{eq:sde01} and \eqref{eq:ifpk01}, respectively.

In this paper, we first present the precise 
problem formulation and its main result on the characterization of the density in Section 2. 
%$$Z^{N}_{t}:= \frac 1 {L_{t}} \sum_{i = 1}^{N} X^{i, N}_{t} I_{(t, \infty)}(\tau^{i,N}).$$
In Section 3, we present the detailed proof, which is mainly 
based on the argument of the Leray-Schauder fixed point theorem.
Section 4 is the summary, and the last section is Appendix for some useful facts collected from existing literatures.

\section{Motivation, problem setup and main results}

\subsection{Motivation}

The role of FPK equation in Mean Field Game theory can be illustrated 
via the following scenario. 
\begin{itemize}
 \item Given a particle system of population size $N$, 
 our interest is the evolution of the mean field term given by
\begin{equation}
 \label{eq:YN01}
Z^{N}_{t}
:= \frac 1 N \sum_{i = 1}^{N} (X^{i, N}_{t})^{p}.
\end{equation} 
In the above, we only consider $p$ as a given positive integer.
If $p=1$, then the mean field term $Z_{t}^{N}$ may refer to the population $p$th mean.
Suppose
 the position $X_{t}^{i, N}$ of the $i$th particle
 follows the dynamics
 \begin{equation}
 \label{eq:mvsde01v01}
dX_{t}^{i,N} =  b(X^{i,N}_{t}, Z^{N}_{t}) dt + dW^{i}_{t};\ \ X_{0}^{i,N} \sim m_{0}, \ i = 1, 2, \ldots, N.
\end{equation}
 driven by an independent 
Brownian motion $W^{i}$ 
with i.i.d. initial distribution $X_{0}^{i, N} \sim m_{0}$,
then one could solve a system of $N$-equations of \eqref{eq:mvsde01v01} to track the 
mean field term  $Z_{t}^{N}$, which bears high computation cost as $N$ becomes large.

On the other hand, if $N$ is a large number,  the
law of large number implies that 
$Z^{N}_{t}$ can be effectively approximated by a deterministic process 
$\int_{\mathbb R^{d}} x^p m(x, t) dx$ for large number $N$ in the sense that
$$\lim_{N\to \infty}Z^{N}_{t} = 
\int_{\mathbb R^{d}} x^{p} m(x, t) dx, \quad \hbox{ almost surely } \forall t>0. $$
In the above, the function $m(\cdot, t)$ is the density of $X_{t}$ of \eqref{eq:mvsde01} 
at its continuum limit, or equivalently is the solution of FPK \eqref{eq:fpk03}. 
Indeed, according to 
the application of the Hewitt and Savage Theorem (see Theorem 5.14 of 
\cite{Car13}), the
FPK \eqref{eq:fpk03} can be used for a fairly broadly defined class of 
symmetric functionals in a large system at its continuum.
\end{itemize}

A similar argument may also be resorted to for a partially killed $N$-particle system: 
\begin{itemize}
 \item Suppose the process $X^{i,N}$ of the $i$th particle following 
 MV-SDE \eqref{eq:mvsde01v01}  has an open unit ball $B_{1}$ as its 
 state space. This means that 
$X^{i, N}$ explodes (absorbed) at
the first exit time $\zeta^{i, N}$ from $B_{1}$.
If we denote the population at time $t$ 
by $$L_{t} := \sum_{i = 1}^{N} I_{(t, \infty)}(\zeta^{i,N}),$$
then the $L_{t}$ is monotonically decreasing from 
the initial size $L_{0} = N$
to $0$ as $t$ goes to infinity.
If we again consider
the mean field term, then the law of large number implies that  
\begin{equation}
\label{eq:YN02}
Y^{N}_{t}
:= \frac 1 {N} \sum_{i = 1}^{N} (X^{i, N}_{t})^{p} I_{(t, \infty)}(\zeta^{i,N})
 \to \int_{B_{1}} x^{p} m(x, t) dx,
\quad \hbox{ almost surely } \forall t>0.
\end{equation}
In the above, 
$m(\cdot, t)$ is the density of killed generic process $X_{t}$ in $B_{1}$ with its absorbing boundary.
\end{itemize}

Therefore, a characterization of the density $m$ in terms of 
its associated FPK is desirable for the killed 
process, however it's not available in the literature to the best of authors' knowledge.

\subsection{Problem setup}
$W$ is a $\mathbb R^{d}$-valued Brownian motion on a filtered probability space 
$(\mathbb P, \Omega, \mathcal F, \{\mathcal F_{t}\}_{t\ge 0})$.
Let $B_{1}$ be the open 
unit ball in $\mathbb R^{d}$ and $p$ be a positive integer. 
We consider MV-SDE of the form \eqref{eq:mvsde01} with the only difference on the state space $B_{1}$. 
To emphasize its state space, we
write the state space $B_{1}$
after semicolon together with given initial state $X_{0}$, 
\begin{equation} \label{eq:sde01}
d X_{t} =  (b(X_{t}, \mathbb E[X_{t}^{p}; B_1]) dt + dW_{t}) I_{B_1}(X_t);
 \quad X_{0} \sim m_{0}.
\end{equation}
In the above, the function $b: B_{1} \times B_{1} \ni (x, y)\mapsto b(x,y) \in \mathbb R^{d}$ 
is a given drift and $X_{0}: \Omega \mapsto B_{1}$ 
is a given  $\mathcal F_{0}$-measurable its initial state with its density $m_{0}$ on $B_{1}$.
Moreover, the mean field term in the 
drift function is understood 
as   $$\mathbb E[X_{t}^{p}; B_1] := \mathbb E[X_{t}^{p} I_{B_{1}}(X_{t})]
= \int_{B_{1}} x^{p} m(x, t) dx.$$

Note that, the boundary $\partial B_{1}$ is set
to be the  absorbing boundary, i.e. once $X_{t}$ reaches the 
cemetery $\partial B_{1}$, the term $I_{B_1}(X_t)$ leads to $dX_t = 0$ afterwards 
and it never returns to $B_{1}$. 
As in the convention, we refer 
$$\zeta = \inf\{t>0: X_{t} \notin B_{1}\}$$ 
as the lifetime of $X$.
%, and assign $X_{t}$ for all $t\ge \zeta$ a cemetery symbol $\o$, 
%i.e. $X_{t}(\omega) = \o$ for all $t\ge \zeta (\omega)$. 
We are interested in
\begin{enumerate}
 \item the existence and uniqueness of the solution of
 MV-SDE \eqref{eq:sde01} up to the lifetime $\zeta$, and further 
 \item its density evolution $m(x, t) = \mathbb P(X_{t} \in dx)/dx$ for 
$(x, t) \in B_{1} \times \mathbb R^{+}$, if it exists.
\end{enumerate}
The above questions are well studied for MV-SDE
if the state $X_{t}$ is valued in the whole space $\mathbb R^{d}$.
However, if the state space is the bounded set $B_{1}$, one shall 
note that $\mathbb P(X_{t} \in B_{1}) <1$ for $t>0$, hence the process
$X_{t}$ has to be a submarkovian
(see the definition of submarkvoian in Page 9 of \cite{CW05}). Indeed, due to the positive probability of the 
explosion on any time interval $(0, t)$, one can obtain, 
$$\int_{B_{1}} m(x, t) dx = \mathbb P(X_{t}\in B_{1}) = 1 - \mathbb P(X_{t} \in \partial B_1) < 1, \ \forall t>0,$$
as long as the solution $X$ of
\eqref{eq:mvsde01} exists.
To proceed, let us be precise with the definition of the solution of \eqref{eq:sde01}.
\begin{definition}\label{eq:d01}
Given a filtered probability space 
$(\mathbb P, \Omega, \mathcal F, \{\mathcal F_{t}\}_{t\ge 0})$
with a $\mathbb R^{d}$-valued Brownian motion $W$, 
and a $B_{1}$-valued 
$\mathcal F_{0}$-measurable random variable $X_{0}$ 
for its initial state,  a process 
$X$ is said to be the solution of \eqref{eq:sde01} (up to its explosion)
with its state space $B_{1}$, if 
%the following identity holds:
 %$$X_{t} = \tilde X_{t} \cdot I_{B_{1}}(\tilde X_{t}) + \o \cdot I_{B_{1}^{c}}(\tilde X_{t}), $$
 %where 
 $(X, \beta)$ satisfies both
\begin{equation}
 \label{eq:sde03}
 d X_{t} = (b(X_{t}, \beta_{t}) dt + dW_{t}) \cdot I_{B_{1}}(X_{t}), 
 \quad X_{0} \sim m_{0},
\end{equation}
and 
\begin{equation}
 \label{eq:beta01}
 \beta_{t} = \mathbb E[X^{p}_{t} I_{B_{1}}(X_{t})].
\end{equation}
\end{definition}

For a two-variable function 
$m: \mathbb R^{d} \times  \mathbb R \ni (x, t) \mapsto m(x, t) \in \mathbb R$, we often treat $m_{t}(x) = m(x, t)$ as a function-valued evolution in time.
If $\beta_{t}$ were given by a known deterministic process in \eqref{eq:sde03} in priori, then
one can directly write its density evolution as of FPK \eqref{eq:fpk02}
with the substitution $\bar b( x, t) = b(x, \beta_{t})$. At least heuristically, 
one can next use \eqref{eq:beta01}  to 
replace $\beta_{t}$ in FPK  \eqref{eq:fpk02} by $\int_{B_{1}} x^{p} m_{t}dx$, 
and obtain a new FPK of the following form for its density evolution:
\begin{equation}
 \label{eq:ifpk01}
 \left\{ 
\begin{array}{ll}
 \partial_{t} m =\frac 1 2 \Delta m - {\rm div}_{x} ( m_{t}(x) \ 
 b(x, \int_{B_{1}} x^{p} m_{t}(x)dx) ), \ &
 \hbox{ on }  B_{1} \times (0, \infty); 
\\ m(x, 0) = m_{0}(x), \ & \hbox{ on }  \bar B_{1};
\\
m(x, t) = 0, \ & \hbox{ on } \partial B_{1} \times (0, \infty).
\end{array}
\right.
\end{equation}
In the above \eqref{eq:ifpk01}, the divergence term shall reads
$${\rm div}_{x} (m_{t}(x) b(x, y)) = \sum_{i= 1}^{d} \partial_{x_{i}} (m_{t}(x) b^{i} 
(x, y)), \hbox{ where } y = \int_{B_{1}} x^{p} m_{t}(x) dx.$$

\subsection{Main result}
To present our main results, we shall impose the following regularity assumptions throughout the text:
\begin{itemize}
 \item [(A1)] $b\in C^{1 + \gamma} (B_{1}^{2}; \mathbb R^{d})$ and 
 $m_{0} \in C_{0}^{2+\gamma}(B_{1})$ for some $\gamma \in (0, 1]$. 
\end{itemize}
In the above assumption, we adopt the notion 
$C^{k+ \gamma}(B_{1}^{2}; \mathbb R^{d})$ 
from \cite{Kry96}
to denote the collection of all functions 
$f: B_{1}^{2}\mapsto \mathbb R^{d}$ which has 
all of their $k$th derivatives are 
$\gamma$-H\"older continuous.
By $C_{0}^{k+ \gamma}(B_{1})$, we refer the space of
functions $f:B_{1} \mapsto \mathbb R$ in $C^{k+ \gamma}(B_{1}, \mathbb R)$ which is smoothly vanishing to zero outside of its domain, see more details about H\"older space in Definition \ref{d:hs01} of
Appendix.

\begin{theorem}
 \label{t:main}
 If we assume (A1), then there exists a solution of MV-SDE 
 \eqref{eq:sde01}, 
 %which has its density solves 
 whose density satisfies FPK \eqref{eq:ifpk01}.
\end{theorem}

\subsection{An example}
Next, we provide a special case of Theorem \ref{t:main}, where the statement of Theorem \ref{t:main} could be easily verified from the symmetry of the initial density directly using Fourier series. We relegate the proof for the general case to Section \ref{sec:3}.

We consider 1-d 
MV-SDE \eqref{eq:sde01} with the drift and initial density given by
\begin{equation}
 \label{eq:e01}
b(x, y) = y^{2}, \ m_{0}(x) = \kappa e^{1/(x^{2}-1)} I_{(-1, 1)} (x), 
\end{equation}
where
$\kappa$ is the normalization constant
$$\kappa = \Big(\int_{-1}^{1} e^{1/(x^{2}-1)} dx \Big)^{-1}.$$
It is noted that both $b$ and $m_{0}$ satisfy (A1).
We use the following $L^{2}(-1,1)$-orthogonal basis: 
for all natural number $n$
$$\eta_{n} (x) = \sin (n\pi(x+1)/2), \ \forall x\in (-1,1).$$
By \cite{str08}, the equation
\begin{equation}
 \label{eq:heat}
 \left\{ 
\begin{array}{ll}
 \partial_{t} m =\frac 1 2 \partial_{xx} m  \ &
 \hbox{ on }  (-1,1) \times (0, \infty); 
\\ m(x, 0) = m_{0}(x), \ & \hbox{ on }   [-1,1];
\\
m(\pm 1, t) = 0, \ & \hbox{ on }  (0, \infty)
\end{array}
\right.
\end{equation}
has the unique solution
$$m(x, t) = \sum_{n\in \mathbb N} (m_{0}, \eta_{n}) 
e^{-n^{2}\pi^{2}t/8} \eta_{n}.
$$
By \cite{Fel71}, the above function $m$ given by Fourier series 
is also the 
density function of a Brownian motion with initial distribution $m_{0}$  absorbed at $\{\pm 1\}$, i.e.
\begin{equation} \label{eq:e02}
d X_{t} =  dW_{t} \cdot I_{(-1,1)}(X_t);
 \quad X_{0} \sim m_{0}.
\end{equation}
If $n$ is even, then $(m_{0}, \eta_{n}) = 0$, since 
$m_{0}$ is even and $\eta_{n}$ is odd.
Therefore, $v$ is an even function of the form
\begin{equation}
 \label{eq:m01}
 m(x, t) = \sum_{n = odd} (m_{0}, \eta_{n}) 
e^{-n^{2}\pi^{2}t/8} \eta_{n}.
\end{equation}
Hence, if $p$ is odd, we have
$$\beta_{t} = \mathbb E[X^{p}_{t} I_{(-1,1)}(X_{t})] = 0, \forall t\ge 0.$$
Moreover, we observe that, due to the fact of 
$b(x, \beta_{t}) = b(x, 0) = 0$, \eqref{eq:heat} and \eqref{eq:e02} 
are equivalent to 
 \eqref{eq:ifpk01} and \eqref{eq:sde01}, respectively. 
 Thus we conclude that
\begin{itemize}
 \item With the setup \eqref{eq:e01}, if $p$ is an odd number, then 
 the density function of \eqref{eq:sde01} has an 
 explicit form \eqref{eq:m01}, and 
  solves \eqref{eq:ifpk01}.
\end{itemize}
We also observe that, the density $m(x, t)$ goes to zero function as $t\to 0$, and zero function is actually stationary distribution of such a process. It's not hard to see this example has the following extensions: 
%\begin{itemize}
%\item 
If there exists some $\gamma\in (0,1]$ such that, 
\begin{itemize}
\item $b\in C^{1	+\gamma}(B_1^2; \mathbb R^d)$ with $b(x, 0) =0$ for all $x$;
\item $m_0 \in C_0^{2+\gamma}(B_1)$ is an even function,
\end{itemize}
then the density function of \eqref{eq:sde01} 
  solves \eqref{eq:ifpk01}. Furthermore, $m$ has a representation via Fourier series, which goes to zero function as $t\to 0$.
%\end{itemize}

\section{Proof of the main result} \label{sec:3}
In this section, we will prove Theorem \ref{t:main}.
We outline the main proof of Theorem \ref{t:main} in Section \ref{s:3-1}
based on some estimation results, whose proof will be provided in Section \ref{s:3-2}. We will use $K$ for a generic constant, and 
$K(\alpha, \beta)$ indicates its dependence on $\alpha$ and $\beta$.

\subsection{Definition of the mapping $\mathcal T$ and the proof of Theorem \ref{t:main}} \label{s:3-1}
We fix an arbitrary $T>0$ and $\gamma \in (0,1)$. 
We also define Banach spaces $\mathcal B$ and $\mathcal R$ given by
 \begin{equation}
 \label{eq:calb01}
  \mathcal B = C^{1/2} ((0, T); \mathbb R^{d}).
\end{equation}
and 
\begin{equation}
 \label{eq:calr01}
 \mathcal R := C^{2 + \gamma, 1+ \frac \gamma 2} ( B_{1} \times (0, T); \mathbb R).
\end{equation}
For precise definitions of elliptic and parabolic H\"older spaces, we refer to 
Appendix \ref{sec:holder}.

To proceed, we introduce an operator $\mathcal T: \mathcal B \mapsto \mathcal B$ through the
composition $\mathcal T = \mathcal T_{1} \circ \mathcal T_{2}$, where
$\mathcal T_1: \mathcal B \mapsto \mathcal R$ and $\mathcal T_2: \mathcal R \mapsto \mathcal B$ 
are defined as follows:

\begin{enumerate}
 \item Define $\mathcal T_{1}: \beta \mapsto \mathcal T_{1}(\beta) := m$, where $m$ solves the following equation with a given process $\beta$,
 \begin{equation}
 \label{eq:fpk01}
 \left\{ 
\begin{array}{ll}
 \partial_{t} m =\frac 1 2 \Delta m - {\rm div}_{x} ( m \ b(x, \beta_{t})), \ &
 \hbox{ on } B_{1} \times  (0, T); 
\\ m(x, 0) = m_{0}(x), \ & \hbox{ on } \bar B_{1};
\\
m(x, t) = 0, \ & \hbox{ on } \partial B_{1} \times (0, T).
\end{array}
\right.
\end{equation}

\item Define $\mathcal T_{2} : m \mapsto \mathcal T_{2}(m)$, where 
\begin{equation}
 \label{eq:tau02}
\end{equation}
$$\mathcal T_{2} (m) = \int_{B_{1}} x^{p} m_{t}(x) dx.$$
\end{enumerate}

%\item We take $\mathcal B$ as 

\begin{proof}
 (of Theorem \ref{t:main})
 
By Lemma \ref{l:T}, 
$\mathcal T$ is a mapping from a Banach space $\mathcal B$ to itself.
Furthermore, the mapping
$\mathcal T: \mathcal B \mapsto \mathcal B$ has the following properties:
\begin{enumerate}
 \item $\mathcal T$ is a continuous compact mapping by Lemma \ref{l:cc01};
 \item $\{x\in \mathcal B : x = \lambda \mathcal T x, \ \lambda \in [0, 1]\}$ is 
 bounded in $\mathcal B$ Lemma \ref{l:fpt02}.
\end{enumerate}
By the
Leray-Schauder's fixed point theorem (FPT), see Theorem 11.2 of \cite{GT01}, 
the mapping $\mathcal T$
 has a fixed point in $\mathcal B$, i.e.
there exists $\hat \beta \in \mathcal B$, such that
$$\mathcal T \hat \beta = \hat \beta.$$
If we set $\hat m = \mathcal T_{1} \hat \beta$, then 
the pair $(\hat m, \hat \beta)$ solves FPK \eqref{eq:ifpk01} by its definition of $\mathcal T = \mathcal T_{1} \circ \mathcal T_{2}$ in \eqref{eq:fpk01} -\eqref{eq:tau02}.
Also due to Proposition \ref{p:ke01}, since $\hat m_{t}$ is the unique solution of \eqref{eq:fpk01}, 
it is the density of the unique solution $\hat X_{t}$
of \eqref{eq:sde03} with given $\beta = \hat \beta$. Together 
with the definition of $\mathcal T_{2}$ given by \eqref{eq:tau02} , the process $\hat \beta_{t}$ 
satisfies  \eqref{eq:beta01}. Therefore, the pair $(\hat X, \hat \beta)$ solves
\eqref{eq:sde03}-\eqref{eq:beta01}, and this gives the solvability of 
 \eqref{eq:sde01} in the sense of Definition \ref{eq:d01} .

\end{proof}

\subsection{Estimates for the mapping $\mathcal T$} \label{s:3-2}
In the proof of Theorem \ref{t:main}, we have used 
 Lemma \ref{l:T}, \ref{l:cc01}, \ref{l:fpt02} concerning the mapping
  $\mathcal T$, and we will present their proofs in this section separately.
First,
we shall verify that $\mathcal T$ is a well defined mapping. This includes
\begin{enumerate}
 \item the unique solvability of FPK \eqref{eq:fpk01}; 
 \item 
 Justification of the set $\mathcal R$ satisfying 
 $\mathcal T_{1}(\mathcal B) 
 \subset \mathcal R \subset \mathcal T_{2}^{-1}(\mathcal B)$.
\end{enumerate} 
 
\begin{lemma}
\label{l:T1}
 $\mathcal T_{1}: \mathcal B \mapsto \mathcal R$ is a well defined mapping
with estimates 
$$|\mathcal T_{1}(\beta)|_{2 + \gamma, 1 + \gamma/2} \le K(|\beta|_{1/2}) |m_{0}|_{2 + \gamma}.$$
\end{lemma}

\begin{proof}
 Rewrite the FPK \eqref{eq:fpk01} into non-divergence form
 $$\partial_{t} m = \frac 1 2 \Delta m - b^{\beta} \circ \nabla m - 
 m \  {\rm div}_{x} (b^{\beta}),$$
 where 
 $$b^{\beta}( x, t) = b( x, \beta(t)).$$
 Note that 
\begin{itemize}
 \item We have $b^{\beta}\in  C^{1.0, \frac 1 2}(B_{1} \times (0, T))$ 
 by the application of 
 Proposition \ref{p:holder01} with (A1) and \eqref{eq:calb01}; Moreover, 
 $$|-b^{\beta}|_{1.0, 1/2} = |b|_{0} + [b^{\beta}]_{1.0, 1/2} \le K |b|_{1.0}( |\beta|_{1/2} + 1).$$
 \item We also have ${\rm div}_{x}(b^{\beta}) \in C^{\gamma, \frac \gamma 2} ( B_{1} \times (0, T))$. Indeed, one can write 
 $${\rm div}_{x} (b^{\beta} (x, t)) = \sum_{i = 1}^{d} \partial_{x_{i}} b^{(i)} (x, \beta_{t})$$
 and use  
 Proposition \ref{p:holder01} once again 
 and the fact that 
 $\partial_{x_{i}} b^{(i)} \in C^{\gamma}(B_{1}^{2})$, which yields
 $$|{\rm div}_{x}(b^{\beta})|_{\gamma, \gamma/2} \le 
 \sum_{i=1}^{d} |\partial_{x_{i}}b^{(i)} (x, \beta_{t})|_{0} + 
 \sum_{i=1}^{d} [\partial_{x_{i}} b^{(i)} (x, \beta_{t})]_{\gamma, \gamma/2} 
 \le K |b|_{1+\gamma} (|\beta|_{1/2} +1).$$
\end{itemize}

Moreover, if we define $m_{0}^{T} (x, t) = m_{0}(x)$, then $m_{0}^{T} \in C^{1 + \frac \gamma 2, 2 + \gamma}( B_{1}  \times (0, T) )$. Therefore, 
by Theorem 10.3.3 of [2], there exists unique solution $m \in C^{ 2 + \gamma, 1 + \frac \gamma 2}( B_{1} \times (0, T)) = \mathcal R$ for \eqref{eq:fpk01}.

If we set $\bar m (x, t) = e^{-\lambda t} m(x, t)$ with $\lambda = |b|_{1+ \gamma}$, then $ {\rm div}_{x}(b^{\beta}) + \lambda \le 0$ and 
$\bar m$ solves 
$$
 \left\{ 
\begin{array}{ll}
 \partial_{t} \bar m = \frac 1 2 \Delta \bar m - b^{\beta} \circ \nabla \bar m - 
 \bar m \  ({\rm div} (b^{\beta}) + \lambda), \ &
 \hbox{ on } B_{1} \times (0, T) ; 
\\ \bar m(x, 0) = m_{0}(x), \ & \hbox{ on } \bar B_{1};
\\
\bar m(x, t) = 0, \ & \hbox{ on }  \partial B_{1} \times (0, T).
\end{array}
\right.
$$
Now we can invoke the estimation from Proposition \ref{p:ke02} to obtain
$$|\bar m|_{2 + \gamma, 1 + \gamma/2} \le K(|b^{\beta}|_{\gamma, \gamma/2}, |{\rm div}_{x}(b^{\beta}) + \lambda|_{\gamma, \gamma/2}) |m_{0}|_{2 + \gamma}$$
which is equivalent to
$$|m|_{ 2 + \gamma, 1 + \gamma/2} \le K(|b|_{1+\gamma}, |\beta|_{1/2}) |m_{0}|_{2 + \gamma}.$$
\end{proof}

Estimation of $\mathcal T_{2}$ directly follows from its definition.
The H\"older space $C^{1.0}((0, T); \mathbb R^{d})$ used below 
is indeed the space of
Lipschitz continuous functions, see the remark on $C^{1.0}$ and $C^{1}$
in Appendix \ref{sec:holder}.
\begin{lemma}
 \label{l:T2}
 $\mathcal T_{2} : \mathcal R \mapsto C^{1.0}((0, T); \mathbb R^{d})$
is well defined with an estimate
$$|\mathcal T_{2}(m)|_{1.0} \le K |m|_{2 + \gamma, 1+ \frac \gamma 2}.$$
%$\Box$
\end{lemma}

\begin{proof}
To prove Lipschitz continuity of the process $\mathcal T_{2}(m)$, we shall 
show 
$$|\mathcal T_{2}(m)(t_{1}) - \mathcal T_{2}(m)(t_{2})| 
\le K |t_{1} - t_{2}|, \ \forall 0 \le t_{1} < t_{2} \le T.$$
This follows from the following estimate for the mapping $\mathcal T_2$ of $m\in \mathcal R$
and $0 < t_{1} \le t_{2} < T$:
$$
\begin{array}
 {ll}
|\mathcal T_{2}(m)(t_{1}) - \mathcal T_{2}(m)(t_{2})| 
& \displaystyle
\le \int_{B_{1}} |x|^{p} |m(x, t_{1}) - m(x, t_{2})| dx
\\ & \displaystyle
= 
\int_{B_{1}} |x|^{p} \cdot \int_{t_{1}}^{t_{2}} |\partial_{t} m| (x, s) ds dx 
\\ & \displaystyle
\le \int_{B_{1}} |x|^{p}  dx \  |m|_{2 + \gamma, 1+ \frac \gamma 2} \ |t_{2} - t_{1}|. 
\end{array}
$$
\end{proof}

Now we can have an estimation of $\mathcal T$ as a well-defined mapping.
\begin{lemma}
 \label{l:T}
 $\mathcal T: \mathcal B \mapsto C^{1.0}((0,T); \mathbb R^{d}) \subset \mathcal B$ is well defined with 
$$|\mathcal T (\beta)|_{1.0} \le K(|\beta|_{\frac 1 2 }) |m_{0}|_{2+\gamma}.$$
\end{lemma}

\begin{proof}
 It is a consequence of Lemma \ref{l:T1} and \ref{l:T2}.
\end{proof}

So far, we established that the operator $\mathcal T$ is well defined from
its domain $\mathcal B$ to itself.
Next, we will prove key facts for the proof of the fixed point theorem: the continuity and the compactness of the operator $\mathcal T$.
For this purpose, we shall briefly recall the following imbedding properties
on H\"older spaces. Consider two H\"older spaces $C^{\gamma}$ and 
$C^{\lambda}$ for $\gamma > \lambda>0$.
Then, $C^{\gamma} \subset C^{\lambda}$ holds and any bounded subset of 
$C^{\gamma}$ is a compact subset of $C^{\lambda}$. Furthermore, 
if (1) $C^{\gamma} \ni \alpha_{n} \to \alpha$ pointwisely; and (2) 
$|\alpha_{n}|_{\gamma} < K$ for any $n\in \mathbb N$, then 
$\alpha_{n} \to \alpha$ in $C^{\lambda}$, 
i.e. $|\alpha_{n} - \alpha|_{\lambda} \to 0$ as $n\to \infty$. 
However, $\alpha_{n}\to \alpha$ in $C^{\gamma}$ (i.e. $|\alpha_{n} - \alpha|_{\gamma} \to 0$)
may not be true.

\begin{lemma}
 \label{l:cc01} $\mathcal T$ is continuous compact in $\mathcal B$.
\end{lemma}

\begin{proof}
 Lemma \ref{l:T} implies that any sequence $\{\beta^{n}\}$ bounded in 
 $\mathcal B = C^{1/2}((0,T); \mathbb R^{d})$ maps to a sequence $\{\mathcal T(\beta^{n})\}$ bounded in $C^{1.0}((0,T); \mathbb R^{d})$, 
 which is precompact in $C^{1/2}((0,T); \mathbb R^{d})$. 
 Thus, $\mathcal T$ is compact.
 
 In this below, we will establish the continuity of the mapping $\mathcal T$. If $\beta^{n} \to \beta^{\infty}$ in 
 $\mathcal B$, we denote, for simplicity
 $$m^{n} = \mathcal T_{1} (\beta^{n}), \ \alpha^{n} = \mathcal T_{2} (m^{n}), \  \forall n \in \mathbb N \cup \{\infty\}.$$
 Our objective is to show $\alpha^{n} \to \alpha^{\infty}$ in $\mathcal B$.
 Since 
 $\{m^{n}:  n\in \mathbb N\}$ is bounded in 
 $C^{2 + \gamma, 1+ \frac \gamma 2} ( B_{1} \times (0, T); \mathbb R)$ 
 by Lemma \ref{l:T1}, the pointwise convergence of $m^{n}$ as $n\to \infty$ holds from stability of the viscosity solution, i.e.
 $$m^{n}(x, t) \to m^{\infty}(x, t), \ \forall (x, t) \in B_{1} \times (0, T).$$
 Since $\{m^{n}:  n\in \mathbb N\}$ is bounded in 
 $C^{2 + \gamma, 1+ \frac \gamma 2} ( B_{1} \times (0, T); \mathbb R)$ by Lemma \ref{l:T1}, 
 one can use bounded convergence theorem to obtain, 
 $$\alpha^{n}(t) - \alpha^{\infty}(t) = \int_{B_{1}} x^{p} \cdot (m^{n}(x, t) - m^{\infty}(x, t)) dx \to 0, \quad \forall t\in (0, T).$$
 Thus, $\alpha^{n}$  converges $ \alpha^{\infty}$ pointwisely. But we know
 $\{\alpha^{n}: n\in \mathbb N\}$ are bounded in $C^{1.0}((0, T); \mathbb R^{d})$ by Lemma \ref{l:T}. Pointwise convergence and boundedness in $C^{1.0}((0,T); \mathbb R^{d})$ implies convergence in $C^{1/2}((0,T); \mathbb R^{d})$, i.e. $\alpha^{n} \to \alpha^{\infty}$ in $\mathcal B$.
 
\end{proof}

The following is the last piece required to  complete the FPT.

\begin{lemma}
\label{l:fpt02} 
The set
$\{\beta: \beta = \lambda \mathcal T \beta, \lambda \in [0, 1]\}$ is bounded in $\mathcal B$.
\end{lemma}

\begin{proof}
 We shall show that, there exists $M>0$, s.t. if $\beta$ solves  \eqref{eq:mvfpk01} below for some $\lambda \in [0, 1]$, then $|\beta|_{C^{1/2}((0, T); \mathbb R^{d})} < M$.
 
\begin{equation}
 \label{eq:mvfpk01}
 \left\{
\begin{array}
 {ll}
 \partial_{t} m = \frac 1 2 \Delta m - {\rm div}_x (b(x, \beta_{t}) m) , & B_{1} \times (0, T)\\
 m(x, 0) = m_{0}(x), & x\in  \bar B_{1} \\
 m(x, t) = 0, & \partial B_{1} \times (0, T)\\
 \beta_{t} = \lambda \int_{B_{1}} x^{p} m(x, t) dx & t\in (0, T).
\end{array}
\right.
\end{equation}
If $\lambda = 0$, then $\beta = 0$, then the conclusion trivially holds.
If $\lambda \in (0, 1]$, we first have
$$|\beta|_{0}  \le \sup_{t\in (0, T)} \lambda \int_{B_{1}} |x|^{p} m(x, t) dx 
\le \sup_{t\in (0, T)} \int_{B_{1}}  m(x, t) dx \le 1.
$$
So, it's enough to show the boundedness of $[\beta]_{1/2}$. To proceed, 
we write the stochastic representation by Proposition \ref{p:ke01}
as follows.
$$\left\{
\begin{array}
 {ll}
 d X_{t} = b(X_{t}, \beta_{t}) dt + dW_{t}, 
 \\
 \tau = \inf \Big\{ t>0: X_{t} \notin B_{1} \Big\}
 \\
 \beta(t) = \lambda \mathbb E^{m_{0}} \Big[ |X_{t}|^{p} 
 I_{[0, \tau)}(t) \Big].
\end{array}
\right.
$$
Without loss of generality, 
we set $0 < t_{1} < t_{2} < T$. Then, we have
$$ 
\begin{array}
 {ll}
 |\beta(t_{1}) - \beta(t_{2})|^{2} 
 & = \displaystyle
 \lambda^{2} \Big |\mathbb E^{m_{0}} [X_{t_{1}} 
 %I_{[t_{1}, 1]} (\tau)
 I_{[0, \tau)}(t_{1})
 ]
 - \mathbb E^{m_{0}} [X_{t_{2}} 
 %I_{[t_{2}, 1]} (\tau)
 I_{[0, \tau)}(t_{2}) 
 ] 
 \Big |^{2} 
 \\& \displaystyle
  \le 2 \lambda^{2} \Big |\mathbb E^{m_{0}} [X_{t_{1}} I_{[t_{1}, t_{2}]} (\tau)] \Big |^{2}
 + 2 \lambda^{2}  \Big|\mathbb E^{m_{0}} [(X_{t_{1}} - X_{t_{2}}) 
 I_{[t_{2}, T]} (\tau)] 
 \Big|^{2}
  \\& \displaystyle
\le 2 \lambda^{2} \Big |\mathbb E^{m_{0}} [ I_{[t_{1}, t_{2}]} (\tau) ] \Big |^{2}
+ 2 \lambda^{2}  \mathbb E^{m_{0}} [|X_{t_{1}} - X_{t_{2}}|^{2} ] 
 \\& \displaystyle
 \le \lambda^{2} K(|b|_{0}) \Big |t_{1} - t_{2} \Big|, 
 \quad \hbox{ by (D9) of \cite{FS06}}.
\end{array}
$$
Therefore, we have
$$[\beta]_{1/2} \le \lambda K^{1/2}(b_{0}).$$
Thus, if we choose $M = 1 + \lambda K^{1/2}(|b_{0}|)$, we shall have
$$|\beta|_{1/2} = |\beta|_{0} + [\beta]_{1/2} \le  1 + \lambda K^{1/2}(|b_{0}|)  = M.$$
\end{proof}

\section{Summary}

In this note, we showed that under assumption  (A1) on the drift function
and the initial density, the killed McKean-Vlasov process \eqref{eq:sde01}
solves FPK \eqref{eq:ifpk01}. We 
observe that uniqueness of the solution to the FPK  \eqref{eq:ifpk01} 
has not been established in this paper  and this will be the subject of 
future work.
If the uniqueness were true, as an application, one can approximate 
the population $p$th mean\eqref{eq:YN02} of large system in a bounded region by 
its associated FPK. It may  also be interesting to consider 
the estimates of the population $p$th moment 
normalized by the number of survivals, 
which is closely related to $Y^{N}$ given by \eqref{eq:YN02}.

\appendix
\section{Appendix}

\subsection{H\"older Space}\label{sec:holder}
For convenience, we introduce the notion of elliptic H\"older space
$C^{k + \gamma}(\mathcal D; \mathcal R)$
 and parabolic 
H\"older space $C^{\frac \gamma 2, \gamma}(\mathcal D \times (0, T); \mathcal R)$ from \cite{Kry96}.  
In this paper, the domain $\mathcal D$ 
may be $B_{1}$, $B_{1}^{2}$ or $(0, T)$; 
and the range $\mathcal R$ may be $\mathbb R$ or $\mathbb R^{d}$. 
If the range $\mathcal R$ is $\mathbb R$, then $\mathcal R$ may be abbreviated, for instance, $C^{\gamma}(\mathcal D)$ means $C^{\gamma}(\mathcal D; \mathbb R)$. 
By uniform continuity, we shall treat $C^{\gamma}(\mathcal D)$ as the same as $C^{\gamma}(\bar {\mathcal D})$ by natural bijective isometry. 
\subsubsection{Elliptic H\"older 
Space}

Let $\mathcal D$ be a domain in $\mathbb R^{d}$ and $\mathcal R$ be a range in $\mathbb R^{d_{1}}$. 
For $u: \mathcal D \mapsto \mathcal R$, 
we define a uniform norm by $|u|_{0} = \sup_{\mathcal D} |u|$, and
we denote by $\partial^{\alpha_{i}}_{x_{i}} u$
the $\alpha_{i}$-th order partial derivative in the variable $x_{i}$, if it exists.
For multiindex $\alpha = (\alpha_{i}: i = 1, \ldots d)$, 
we use $D^{\alpha} u = \partial_{x_{1}}^{\alpha_{1}} \cdots \partial_{x_{d}}^{\alpha_{d}} u$.

For 
$k\in \mathbb N \cup \{0\}$, we denote by $C^{k}_{loc} (\mathcal D, \mathcal R)$ the set of all functions $u:\mathcal D \mapsto \mathcal R$ whose derivatives $D^{\alpha} u$ for $|\alpha| \le k$ are continuous in $\mathcal D$. One can define a norm in $C^{k}_{loc}(\mathcal D, \mathcal R)$ by
$$|u|_{k} = \sum_{i = 0}^{k} \max_{|\alpha| = i} |D^{\alpha} u |_{0}.$$
Then the functions $u$ having finite norm consists of Banach space, and we refer it to $C^{k}(\mathcal D, \mathcal R)$. For instance, $u = e^{x}:\mathbb R \mapsto \mathbb R$ belongs to $C^{k}_{loc}(\mathbb R, \mathbb R)$ but not $C^{k} (\mathbb R, \mathbb R)$.

For $\gamma \in (0,1]$, we can also define a H\"older seminorm for a function $u\in C(\mathcal D, \mathbb R)$ by, 
$$[u]_{\gamma} = \sup_{x,y\in \mathcal D, x \neq y} 
\frac{|u(x) - u(y)|}{|x - y|^{\gamma}}.$$

\begin{definition}
 \label{d:hs01}
 For a decimal number $\gamma \in (0, 1]$ and an integer
 $k \in \mathbb N \cup \{0\}$,  H\"older space
 $C^{k + \gamma}(\mathcal D, \mathcal R)$ is the Banach space of all 
 functions $u \in C^{k}(\mathcal D, \mathcal R)$ for which the norm
 $$|u|_{k + \gamma} = |u|_{k} + \max_{|\alpha| = k} [D^{\alpha} u]_{\gamma}$$
 is finite. 
 \end{definition}
 In the above, we emphasize that $\gamma$ is a decimal number (writing with decimal point) and $k$ is an integer to avoid the following ambiguity. Note that $C^{1.0}(\mathcal D, \mathcal R)$ is $1$-H\"older space (or Lipschitz continuous space) with a finite norm w.r.t.
 $$|u|_{1.0} =  |u|_{0} + \sup_{x,y\in \mathcal D, x \neq y} 
\frac{|u(x) - u(y)|}{|x - y|^{\gamma}},$$
while $C^{1}(\mathcal D, \mathcal R)$ is a continuous differentiable function space with a finite norm
 w.r.t.
 $$|u|_{1} = |u|_{0} + \max_{i = 1}^{d} |\partial_{x_{i}} u|_{0}.$$
 For instance, 
 $f(x) = |x|$ is in $C^{1.0} ([-1, 1]) \setminus C^{1}  ([-1, 1])$
 with its norm
 $$|f|_{1.0} = |f|_{0} + [f]_{1.0} = 2.$$
 Another example is that $g(x) = x^{2} {\rm sgn}(x)$ is in $C^{2.0}([-1,1] \setminus C^{2}([-1, 1])$ with 
 $$|g|_{2.0} = |g|_{0} + |g'|_{0} + [g']_{1.0} = 5. $$
 In general, $C^{k+1}(\mathcal D, \mathcal R)$ is a proper subset of $C^{k+1.0}(\mathcal D, \mathcal R)$.  
 
Next, we use the extension of $u:\mathcal D\mapsto \mathcal R$
with $\tilde u(x) = u(x) I_{\mathcal D}(x): \mathbb R^{d} \mapsto \mathbb R^{d_{1}}$ by taking values the same as $u$ in $\mathcal D$ otherwise zero. 
\begin{definition}
Let $\mathcal D$ be a bounded set in $\mathbb R^{d}$. The space $C_{0}^{k+\gamma} (\mathcal D, \mathcal R)$ is defined by
 $$C_{0}^{k+\gamma} (\mathcal D, \mathcal R) = 
 \{ u \in C^{k+\gamma} (\mathcal D, \mathcal R):
  u(x) I_{\mathcal D}(x) \in C^{k + \gamma} (\mathbb R^{d}, \mathcal R)\}.$$
\end{definition}
 
\subsubsection{Parabolic H\"older Space}
Let $\mathcal D$ be a domain in $\mathbb R^{d}$, $\mathcal Q = \mathcal D \times (0, T)$ be the parabolic domain in $\mathbb R^{d+1}$ for some $T>0$, and
$\mathcal R$ be a range in $\mathbb R^{d_{1}}$. We are going to define norms for $u: \mathcal Q \mapsto \mathcal R$ in this below.

First, we define parabolic metric on $\mathbb R^{d+1}$: for any 
$z_{1} = (x_{1}, t_{1}), z_{2} = (x_{2}, t_{2}) \in \mathbb R^{d+1}$
$$\rho(z_{1}, z_{2}) =|x_{1} - x_{2}|  +  |t_{1} - t_{2}|^{1/2}.$$ 
Then, we set the parabolic H\"older seminorm for $u\in C(\mathcal Q)$ by,
$\gamma \in (0,1)$
$$[u]_{\gamma, \gamma/2} = \sup_{z_{1}, z_{2} \in \mathcal Q, z_{1} \neq z_{2}} \frac{|u(z_{1}) - u(z_{2})|}{\rho^{\gamma}(z_{1}, z_{2})}.$$

\begin{definition}
 \label{d:hs02}
 For $\gamma \in (0, 1]$ and $k \in \mathbb N \cup \{0\}$, the 
 parabolic H\"older space
 $C^{2k + \gamma, k + \gamma/2}(\mathcal Q, \mathcal R)$ 
 is the Banach space of all 
 functions $u \in C(\mathcal Q, \mathcal R)$ for which the norm
 $$|u|_{2k + \gamma, k + \gamma/2} = |u|_{0} + 
 \sum_{i = 1}^{k} |D^{i}_{t} u|_{0} + \sum_{i = 1}^{2k} 
 \max_{|\alpha| = i} |D_{x}^{\alpha} u|_{0} + 
 \max_{|\alpha| = 2k} [ D_{t}^{k} D_{x}^{\alpha} u]_{\gamma, \gamma/2}.$$
 is finite.
\end{definition}
In this text, we only use $C^{2k + \gamma, k + \gamma/2}(\mathcal Q, \mathcal R)$ for $k = 0, 1$.
We also need the following elementary fact between elliptic 
H\"older and 
parabolic H\"older spaces. 
\begin{proposition}
 \label{p:holder01}
  If
 $f\in C^{\delta} ((0, T); B_{1})$ and $g\in C^{\gamma}(B_{1}^{2})$ be two functions
 for some constants 
 $\delta \in (0, 1/2]$ and $\gamma \in (0, 1]$, then $h(x, t) = g(x, f(t))$ 
 belongs to $C^{2 \delta \gamma, \delta \gamma} ( B_{1} \times (0, T))$
 with 
 $$[ h ]_{2 \delta \gamma, \delta \gamma} \le K [g]_{\gamma} ([f]_{\delta} +1).$$

\end{proposition}

\begin{proof}
 The result follows from the following inequalities:
$$
|h(x_{1}, t_{1}) - h(x_{2}, t_{2})| 
= |g(x_{1}, f(t_{1})) - g(x_{2}, f(t_{2}))|
\le [g]_{\gamma} (|f(t_{1}) -f(t_{2})|^{2} + | x_{1} -  x_{2}|^{2})
$$
Since $|f(t_{1}) -f(t_{2})| \le [f]_{\delta} |t_{1} - t_{2}|^{\delta}$, 
$|t_{1} - t_{2}| + |x_{1} - x_{2}|\le 3$, and $2\delta \le 1$, we conclude that
$$|h(x_{1}, t_{1}) - h(x_{2}, t_{2})|   \le K [g]_{\gamma} ([f]_{\delta} +1) 
(|t_{1} - t_{2}|^{1/2} + |x_{1} - x_{2}|)^{2\delta \gamma}.
$$

\end{proof}

\subsection{Relations between PDEs and SDE without mean field term}

We consider the density of a killed process on the unit ball $B_{1}$ given by
\begin{equation}
 \label{eq:sde02}
dX_{t} = (\bar b(X_{t}, t) dt + dW_{t}) I_{B_{1}}(X_{t}), \  X_{0} \sim m_{0}
\end{equation}
for a function $\bar b\in C^{2 +  \delta, 1 + \frac \delta 2}(B_{1} \times (0, T))$.
One can define a semigroup $\{P_{s,t}: 0 \le s \le t\}$ on a Banach space $C_{0} = C_{0}(B_{1})$ by 
$$P_{s, t} f(x) = \mathbb E^{x, s} [ f(X_{t}) ].$$ 
It can be checked that 
\begin{itemize}
 \item %(Identity) 
 $P_{t, t} = I$;
 \item %(Closure and Associativity) 
 $P_{s, t} P_{t, r} = P_{s,r}$ for $0\le s \le t \le r$;
\end{itemize}
Recall that the generator
$$L_{t} f(x) = \lim_{h} \frac{P_{t, t+h} - I}{h} f(x).$$
In this case, the generator can be written explicitly as
$$ L_{t} f(x) = \bar b \circ \nabla f + \frac 1 2 \Delta f$$
and the domain $\mathcal D(L_{t})$ of the generator includes the smooth test function set $C^{\infty}_{0}(B_{1})$. Moreover, the adjoint operator of $L_{t}$ is given by
$$ L_{t}^{*} f(x) = - {\rm div}_x (\bar b \ f) + \frac 1 2 \Delta f.$$

Formally, if we denote the density of $X_{t}$ on 
$B_{1}$ by $m(t, \cdot)$, i.e. 
$$P_{0, t} f(x) = (m_{t}, f), \ \forall f\in C_{0}.$$
It is also noted that, $X_{t}$ is a submarkovian on $B_{1}$, since $\int_{B_{1}} m_{t}(x) dx = 1 - \mathbb P(X_{t} \in \partial B_{1}) \le 1$.
One can carry out
 $$\frac d {dt} P_{s,t} f(x) = \lim_{h} \frac{P_{s, t+h} - P_{s,t}}{h} f(x) = 
 P_{s, t} L_{t} f(x), \ \forall f\in C_{0}^{\infty}.$$
Taking $s = 0$, it becomes
 $$\frac d{dt} (m_{t}, f) = (m_{t}, L_{t} f), \  \forall f\in C_{0}^{\infty}$$
 which implies the Kolmogorov forward equation with appropriate initial-boundary conditions:
\begin{equation}
 \label{eq:fpk02}
 \left\{
\begin{array}
 {ll}
\partial_{t} m =L^{*}_{t} m, & B_{1} \times (0, T), \\
m(x, 0) = m_{0}(x), & x\in \bar B_{1}, \\
m(x, t) = 0, &  \partial B_{1} \times (0, T).
\end{array}\right.
\end{equation}

 Next, if we denote 
$u(x, t) = P_{t, T} g(x)$ for some fixed $g\in C_{0}$, 
and $u \in C^{1,2}$ for some fixed $T$, then we can write
 $$\frac d {dt} P_{t, r} g(x) = \lim_{h} \frac{ P_{t + h, r} - P_{t, r}}{h} g(x) = 
 \lim_{h} \frac{ I - P_{t, t+h} }{h}  P_{t + h, r} g(x) = 
 - L_{t} P_{t,r} g(x).$$
This implies, by taking $r = T$, the  Kolmogorov backward equation with some terminal-boundary conditions:
\begin{equation}
 \label{eq:kbe01}
 \left\{
\begin{array}
 {ll}
\partial_{t} u + L_{t} u = 0, & B_{1} \times (0, T), \\
u(x, T) = g(x), & x\in \bar B_{1}, \\
u(x, t) = 0, & \partial B_{1} \times (0, T).
\end{array}\right.
\end{equation}

\begin{proposition}
 \label{p:ke01}
\end{proposition}
Assume 
$\bar b \in C^{2 +  \delta, 1 + \frac \delta 2}((0, \infty) \times B_{1})$, 
$m_{0}, g \in C_{0}^{\delta} (B_{1})$. Then, 
\begin{enumerate}
 \item  The density $m(x, t)$ of $X$ of \eqref{eq:sde02} on the open set $B_{1}$ is the unique solution of \eqref{eq:fpk02} in $C^{2 + \delta, 1+\delta/2}$;
 \item $u(x, t) = P_{t, T} g(x)$ is the unique solution of \eqref{eq:kbe01} in $C^{2 + \delta, 1+\delta/2}$.
\end{enumerate}
\begin{proof}
One can first write down non-divergence form of \eqref{eq:fpk02}. Then, 
the uniqueness of the solution and its regularity result of \eqref{eq:fpk02} and \eqref{eq:kbe01} directly follows from Theorem 10.3.3 of [2]. The relation of \eqref{eq:fpk02} and transition density of the submarkovian process $X$ is referred to Section 4.1 of \cite{Pav14}. 
%\footnote{?? The statement of the reference is not quite satisfactory}
The stochastic representation of \eqref{eq:kbe01} to the function $P_{t, T}g(x)$ is referred to 
Section 40.2 of \cite{Bas11}.
\end{proof}

We need the following estimate in the sequel. Consider 
\begin{equation}
 \label{eq:kbe03}
 \left\{
\begin{array}
 {ll}
\partial_{t} u  = L_{t} u + c u, & B_{1} \times (0, T), \\
u(x, 0) = g(x), & x\in \bar B_{1}, \\
u(x, t) = 0, & \partial B_{1} \times (0, T).
\end{array}\right.
\end{equation}

\begin{proposition}
\label{p:ke02} 
\end{proposition}
If $b, c \in C^{\delta, \delta/2}$, $g\in C_{0}^{2 + \delta}$, and $c\le 0$, 
then 
\eqref{eq:kbe03} is uniquely solvable satisfying
$$|u|_{2 + \delta, 1+\delta/2} \le K(|b|_{\delta, \delta/2}, |c|_{\delta, \delta/2}) 
|g|_{2 + \delta}.$$
\begin{proof}
 Unique solvability is implied by Theorem 10.3.3 of \cite{Kry96} and 
 the estimate is given by Theorem 10.2.2 of \cite{Kry96}.
\end{proof}

\bibliographystyle{plain}
\bibliography{/Users/songqsh/Dropbox/R/refs}
\end{document}